\documentclass[11pt]{amsart}
\usepackage{amsthm}
\usepackage{amscd}  
\usepackage{graphicx}
\usepackage{amssymb}
\usepackage{epstopdf}
\theoremstyle{plain}
\newtheorem{Thm}{Theorem}[section]

\newtheorem{Cor}[Thm]{Corollary}
\newtheorem{Lem}[Thm]{Lemma}

\theoremstyle{definition}
\newtheorem{Defn}[Thm]{Definition}

\newtheorem{Rem}[Thm]{Remark}
 
\numberwithin{equation}{section}

\title{Non-commutative deformations of tilting bundles and the derived McKay correspondence}
\author{Dedicated to Professor Dmitri Orlov for his sixtieth birthday \\ \\
Yujiro Kawamata}
\date{}                                           
\usepackage{fancyhdr}
\begin{document}
\maketitle

\begin{abstract}
We prove that the tilting bundle and the derived McKay correspondence extends under formal non-commutative deformations
by using \v Cech cohomology of non-commutative schemes.

14A22, 14D15, 14E16. 
\end{abstract}

%%%%%%%%%%%%%%%%%%%%%%%%%%%%%%%%%%
%%%%%%%%%%%%%%%%%%%%%%%%%%%%%%%%%%
%%%%%%%%%%%%%%%%%%%%%%%%%%%%%%%%%%
\section{Introduction}

We defined twisted non-commutative (NC) deformations of algebraic varieties by gluing NC deformations of algebras
and described the semi-universal deformations by using the Hochschild cohomology groups in the case of smooth varieties in \cite{smooth}.
We constructed NC deformations of rational surface singularities of type $A_n$ explicitly and proved that the 
derived McKay correspondence extends under such deformations in \cite{An}.
In this paper, we treat the extension problem of tilting bundles for general NC deformations and prove that 
the derived McKay correspondence extends in general as long as it is given by a tilting bundle.

A Gorenstein singularity has sometimes a {\em crepant} resolution, a resolution of singularities which has no discrepancy of 
canonical divisors.
Its non-commutative version, {\em NC crepant resolution}, is defined in \cite{VdB}.  
A (commutative) crepant resolution is geometric, and NC crepant resolution is algebraic.
A derived McKay correspondence is an equivalence of derived categories between commutative and NC crepant resolutions.
In many important cases, such a correspondence is given by a {\em tilting bundle}.
For example, a tilting bundle exists when the crepant resolution is given by an equivariant Hilbert scheme.
Indeed a tilting bundle is given by the universal sheaf (\cite{KV}, \cite{BKR}).
On the other hand, we note that Wahei Hara found an example of a crepant resolution of a Gorenstein singularity 
which has no tilting bundle.

We treated only the NC deformations without twist in this paper.
It is reasonable because the vanishing of $H^2(\mathcal O)$, the tangent space corresponding to the twisting deformations, 
holds in many cases.
Moreover we treat only formal NC deformations unlike the case treated in  \cite{An}.
Here we note that NC deformations are usually only formal and not convergent in most cases.  

The category of quasi-coherent modules $\text{Qcoh}(A)$ on a twisted NC scheme $A$ becomes an abelian
category.
Indeed a homomorphism between quasi-coherent modules is a collection of homomorphisms between 
their component right modules over associative algebras, 
and the kernel and cokernel of the homomorphism is a collection of kernels and cokernels
of these homomorphisms (cf. \cite{An}).
Moreover $\text{Qcoh}(A)$ becomes a Grothendieck category (\cite{DLL}, \cite{An}).

In the case of usual commutative deformations, a locally free coherent sheaf $F^0$ extends along deformations
if $\text{Ext}^2(F^0,F^0) = 0$, and uniquely if $\text{Ext}^1(F^0,F^0) = 0$.
These vanishing property is a part of the definition of the tilting bundle, hence the tilting bundles always extend uniquely.
A locally free coherent sheaf $F^0$ on an algebraic variety $X^0$ is called a {\em tilting bundle}
if the following hold: (1) $\text{Ext}_{X^0}^p(F^0, F^0) = 0$ for $p > 0$; 
(2) if $x \not\cong 0$ for $x \in D^-\text{coh}(X^0)$, then $\text{Hom}_{X^0}(F^0, x[p]) \ne 0$ for some $p$.
 
In our NC situation, similar argument works.
The point of the proof is to calculate cohomology groups which is defined by using injective resolutions, 
and it will be done by extending the mechanism of
\v Cech cohomology groups.
Here we note that our module $F$ on an NC scheme is only defined as a presheaf and not a sheaf because of the 
non-commutativity.

The main result is the following:

\begin{Thm}
Let $f^0: X^0 \to \text{Spec}(B^0)$ be a proper morphism from an algebraic variety to an affine variety, and 
let $F^0$ be a tilting bundle on $X^0$.
Let $(R, \frak m)$ be an Artin local $k$-algebra with $R/\frak m = k$, 
and let $A = (A_i,\phi_{ij})$ be an NC deformation of $X^0$ over $R$.
Then there exists, uniquely up to isomorphism, a tilting bundle $F$ on $A$ which is flat over $R$ and 
$R/\frak m \otimes_R F \cong F^0$.
Moreover, for the endomorphism algebra $E = \text{Hom}_A(F, F)$,  
there is a $k$-linear equivalence of triangulated categories:
\[
\Phi: D^b\text{coh}(A) \to D^b\text{mod-}E
\]
given by $\Phi(x) = R\text{Hom}_A(F, x)$ with its quasi-inverse given by $\Psi(y) = y \otimes^{L}_E F$.
\end{Thm}

In \S 2 and \S 3, we recall the definitions of NC schemes, quasi-coherent modules on them and NC deformations (\cite{smooth}).
In \S 4, we construct a \v Cech resolution for a quasi-coherent modules.
This enable us to handle cohomology groups easily.
Then in \S 5, we prove the extendability of locally free modules and tilting bundles along NC deformations, and finally 
prove the extension theorem of the derived McKay correspondence in \S 6.

NC will be an abbreviation of \lq not necessarily commutative' in this paper.
We work over the base field $k$ throughout the paper.

The author would like to thank NCTS of National Taiwan University where 
this work was partly done while the author visited there.
This work is partly supported by JSPS Kakenhi 21H00970.

%%%%%%%%%%%%%%%%%%%%%%%%%%%%%%%%%%
%%%%%%%%%%%%%%%%%%%%%%%%%%%%%%%%%%
%%%%%%%%%%%%%%%%%%%%%%%%%%%%%%%%%%
\section{NC scheme}

We recall the definitions of NC schemes and quasi-coherent sheaves on them.
We only consider the version where there is no twisting.

\begin{Defn}
Let $I$ be a poset such that the minimum exists for any two elements.
We denote by $i \cap j = \min \{i,j\}$.
An {\em NC scheme} $A = (A,\phi) = (A_i, \phi_{ij})$ over $I$ is given by a collection of unital associative NC algebras $A_i$ for $i \in I$, 
and gluing homomorphisms $\phi_{ij}: A_j \to A_i$ for $i < j$, which satisfy the following conditions:

\begin{itemize}
\item (flatness) $\phi_{ij}$ is flat for $i < j$.

\item (birationality) The natural homomorphism $A_i \otimes_{A_j} A_i \to A_i$ given by
$a_i \otimes a_i' \mapsto a_ia_i'$ for $a_i,a_i' \in A_i$ is bijective for $i < j$.

\item (cocycle) $\phi_{ij}\phi_{jk} = \phi_{ik}$ for $i < j < k$.
\end{itemize}
An equivalence between twisted NC schemes $(A,\phi) \sim (A',\phi')$ over $I$ is given by
algebra isomorphisms $f_i: A_i \to A'_i$ 
such that $\phi'_{ij}f_j = f_i\phi_{ij}$.
\end{Defn}

\begin{Defn}
A {\em quasi-coherent module} $M = (M, \psi) = (M_i, \psi_{ij})$ on an NC scheme $A$ over a poset $I$ is a 
collection of right $A_i$-modules $M_i$ for $i \in I$, and
gluing isomorphisms $\psi_{ij}: M_j \otimes_{A_j} A_i \stackrel{\cong}{\longrightarrow} M_i$ for $i < j$, 
which satisfy the following conditions:

\begin{itemize}
\item $\psi_{ij}(mx \otimes y) = \psi_{ij}(m)\phi_{ij}(x)y$ \, for $m \in M_j$, $x \in A_j$ and $y \in A_i$.

\item $\psi_{ij}(\psi_{jk} \otimes_{A_j} A_i) = \psi_{ik}$ \, on \, $M_k \otimes_{A_k} A_j \otimes_{A_j} A_i = M_k \otimes_{A_k} A_i$ 
for $i < j < k$.
\end{itemize}
$M$ is said to be {\em coherent} if all $M_i$ are finitely generated modules.
Moreover if all $M_i$ are free modules of rank $r$, then $M$ is called a {\em locally free}
module of rank $r$.
We note that a locally free module of rank 1 is not invertible, because the $A_i$ are NC.
\end{Defn}

A homomorphism $h: M \to N$ of quasi-coherent modules consists of homomorphisms
$h_i: M_i \to N_i$ as right $A_i$-modules which are compatible with gluing isomorphisms $\psi_{ij}^M$ and $\psi_{ij}^N$.
The category of quasi-coherent modules $\text{Qcoh}(A)$ becomes an abelian category, where 
$\text{Ker}(h)$ and $\text{Coker}(h)$ are given by the collections of 
$\text{Ker}(h_i)$ and $\text{Coker}(h_i)$ thanks to the flatness of the gluing homomorphisms $\phi_{ij}$.
The direct limit $\varinjlim_{\lambda} M(\lambda)$ of modules is also given by the collection of the components
$\varinjlim_{\lambda} M(\lambda)_i$, 
and $\text{Qcoh}(A)$ becomes a Grothendieck category (\cite{DLL}, \cite{An}).

\begin{Rem}
(1) We changed the direction of an equality $i < j$ from \cite{smooth} so that $i < j$ corresponds to $U_i \subset U_j$.

(2) We note that $(A_i, \phi_{ij})$ and $(M_i, \psi_{ij})$ are presheaves instead of sheaves.
\end{Rem}

%%%%%%%%%%%%%%%%%%%%%%%%%%%%%%%%%%
%%%%%%%%%%%%%%%%%%%%%%%%%%%%%%%%%%
%%%%%%%%%%%%%%%%%%%%%%%%%%%%%%%%%%
\section{NC deformation}

We recall the definition of NC deformations.

\begin{Defn}
Let $X^0$ be an algebraic variety over a field $k$ and let $(R,\frak m)$ be an Artin local $k$-algebra
with $R/\frak m = k$.
Let $\{U_i = \text{Spec}(A_i^0)\}_{i \in I}$ be a covering of $X^0$ by affine open subsets, and let 
$\phi_{ij}^0: A_j^0 \to A_i^0$ be the restriction homomorphisms for $U_i \subset U_j$.
An {\em NC deformation} of $X^0$ over $R$ with respect to the covering $\{U_i\}$ 
is a pair $(A, \alpha)$ consisting of an NC scheme $A = (A_i, \phi_{ij})$ and 
a set of isomorphisms $\alpha = (\alpha_i)$ which satisfy the following conditions:

(1) The $A_i$ are flat $R$-algebras and the $\phi_{ij}$ are $R$-homomorphisms.

(2) $\alpha_i: R/\frak m \otimes_R A_i \to A_i^0$ are isomorphisms such that 
$\phi_{ij}^0 \circ \alpha_j= \alpha_i \circ (R/\frak m \otimes_R \phi_{ij}) $ for $U_i \subset U_j$.
\end{Defn}

\begin{Lem}
Let $(A,\alpha)$ be an NC deformation of an algebraic variety $X^0$, and
let $i,j,k,l \in I$ satisfying $i = j \cap k$ and $j,k < l$.
Then the natural right $A_l$-homomorphism $A_j \otimes_{A_l} A_k \to A_i$ given by 
$a_j \otimes a_k \mapsto \phi_{ij}(a_j)\phi_{ik}(a_k)$ for $a_j \in A_j$ and $a_k \in A_k$ is bijective.
\end{Lem}

\begin{proof}
We proceed by induction on the length of $R$.
If $R = k$, then the condition is a consequence of the fact that the $A_i^0$ have a common field of fractions $k(X^0)$.   
Let $0 \to J \to R' \to R \to 0$ be a small extension of Artin local algebras, i.e., $J \subset R'$ is an ideal such that
$\frak m' J = 0$, and let $A' = (A'_i, \phi_{ij}')$ be an NC deformation of $X^0$ over $R'$ such that 
$A \cong R \otimes_{R'} A'$.
Assuming that $A$ satisfies the conclusion of the lemma, we will prove the same statement holds for $A'$.
We have the following commutative diagram
\[
\begin{CD}
0 @>>> J \otimes A_j^0 \otimes_{A_l^0} A_k^0 @>>> A'_j \otimes_{A'_l} A'_k @>>> 
A_j \otimes_{A_l} A_k @>>> 0 \\
@. @VVV @VVV @VVV \\
0 @>>> J \otimes A_i^0 @>>> A'_i @>>> A_i @>>> 0,
\end{CD}
\]
where the first line is derived from an exact sequence
\[
0 \to J \otimes A_j^0 \to A'_j \to A_j \to 0 
\]
by tensoring $\otimes_{A'_l} A'_k$ with the isomorphisms
$A_j \otimes_{A'_l} A'_k \cong A_j \otimes_{A_l} A_k$ and 
$A_j^0 \otimes_{A'_l} A'_k \cong A_j^0 \otimes_{A_l^0} A^0_k$. 
It is exact thanks to the flatness of $\phi'_{kl}$.
Since both left and right vertical arrows are bijective, so is the middle one.
\end{proof}

%%%%%%%%%%%%%%%%%%%%%%%%%%%%%%%%%%
%%%%%%%%%%%%%%%%%%%%%%%%%%%%%%%%%%
%%%%%%%%%%%%%%%%%%%%%%%%%%%%%%%%%%
\section{\v Cech cohomology}

For an NC scheme $A = (A, \phi)$, 
we will assume the following additional conditions from now on:

\begin{itemize}
\item The index set $I$ is finite.

\item The natural right $A_l$-homomorphism $A_j \otimes_{A_l} A_k \to A_i$ given by 
$a_j \otimes a_k \mapsto \phi_{ij}(a_j)\phi_{ik}(a_k)$ is bijective 
for $i,j,k,l$ with $i = j \cap k$, $j,k < l$, $a_j \in A_j$ and $a_k \in A_k$.
\end{itemize}

For $s \in I$ and a right $A_s$-module $M_s$, 
we define $[M_s] \in \text{Qcoh}(A, \phi)$ in the following way.
We define 
\[
[M_s]_i = M_s \otimes_{A_s} A_p
\]
for $p = i \cap s$.
We note that $[M_s]_i$ has a right $A_i$-module structure because $A_p$ has a $A_i$-module structure
given by $\phi_{pi}$.

For $i < j$, let $q = j \cap s$. 
Then we will define a gluing isomorphism $\psi_{ij}^{[M_s]}: [M_s]_j \otimes_{A_j} A_i \to [M_s]_i$, i.e., 
\[
\psi_{ij}^{[M_s]}: (M_s \otimes_{A_s} A_q) \otimes_{A_j} A_i \to M_s \otimes_{A_s} A_p
\]
by 
\[
m_s \otimes a_q \otimes a_i \mapsto m_s \otimes \phi_{pq}(a_q) \phi_{pi}(a_i)
\]
for $m_s \in M_s$, $a_q \in A_q$ and $a_i \in A_i$.

\begin{Lem}
(1) $\psi_{ij}^{[M_s]}$ is a well-defined isomorphism of right $A_i$-modules.

(2) $[M_s] \in \text{Qcoh}(A, \phi)$.

(3) There is a natural homomorphism $r^M_s: M \to [M_s]$ in $\text{Qcoh}(A, \phi)$.

(4) The functor $[\,\,\,]: \text{Mod-}A_s \to \text{Qcoh}(A, \phi)$ is exact.

(5) $[M_{i \cap j}] = [[M_i]_j]$ and $r_j^{[M_i]} r_i^M = r_{i \cap j}^M$.
\end{Lem}

\begin{proof}
(1) It is well defined because for $a_s \in A_s$ and $a_j \in A_j$  
\[
\begin{split}
&\psi_{ij}^{[M_s]}(m_sa_s \otimes a_q \otimes a_i)
= m_s a_s \otimes \phi_{pq}(a_q) \phi_{pi}(a_i) \\
&= m_s \otimes \phi_{ps}(a_s) \phi_{pq}(a_q) \phi_{pi}(a_i)
= \psi_{ij}^{[M_s]}(m_s \otimes \phi_{qs}(a_s) a_q \otimes a_i), \\
&\psi_{ij}^{[M_s]}(m_s \otimes a_q \phi_{qj}(a_j) \otimes a_i) 
= m_s \otimes \phi_{pq}(a_q) \phi_{pj}(a_j)\phi_{pi}(a_i) \\
&= m_s \otimes \phi_{pq}(a_q) \phi_{pi}(\phi_{ij}(a_j)a_i)
= \psi_{ij}^{[M_s]}(m_s \otimes a_q \otimes \phi_{ij}(a_j) a_i).
\end{split}
\]
The bijectivity is a consequence of the following:
\[
(M_s \otimes_{A_s} A_q) \otimes_{A_j} A_i
\cong M_s \otimes_{A_s} (A_q \otimes_{A_j} A_i) 
\cong M_s \otimes_{A_s} A_p. 
\]

(2) We check the cocycle condition for $\psi_{ij}^{[M_s]}$.
For $r= k \cap s$, we check
\[
\begin{split}
&\psi_{ij}^{[M_s]}\psi_{jk}^{[M_s]} = \psi_{ik}^{[M_s]}: 
((M_s \otimes_{A_s} A_r) \otimes_{A_k} A_j) \otimes_{A_j} A_i \\
&\to (M_s \otimes_{A_s} A_q) \otimes_{A_j} A_i
\to M_s \otimes_{A_s} A_p.
\end{split}
\]
Indeed
\[
\begin{split}
&m_s \otimes a_r \otimes a_j \otimes a_i \mapsto m_s \otimes \phi_{qr}(a_r) \phi_{qj}(a_j) \otimes a_i \\
&\mapsto m_s \otimes \phi_{pr}(a_r) \phi_{pj}(a_j) \phi_{pi}(a_i) = m_s \otimes \phi_{pr}(a_r) \phi_{pi}(\phi_{ij}(a_j)a_i)
\end{split}
\]
for $m_s \in M_s$, $a_r \in A_r$, $a_j \in A_j$ and $a_i \in A_i$.

(3) It is given by natural homomorphisms $M_i \to M_s \otimes_{A_s} A_p \cong M_p$ as right $A_i$-modules.
Indeed it is given by $m_i \mapsto \psi_{pi}(m_i \otimes 1)$.

(4) Let $0 \to L_s \to M_s \to N_s \to 0$ be an exact sequence of right $A_s$-modules.
Since $A_p$ is flat over $A_s$, we obtain an exact sequence 
$0 \to [L_s]_i \to [M_s]_i \to [N_s]_i \to 0$, hence $0 \to [L_s] \to [M_s] \to [N_s] \to 0$ is exact.

(5) is clear.
\end{proof}

\begin{Lem}\label{local}
For $M, N \in \text{Qcoh}(A, \phi)$, there are natural isomorphisms
\[
\text{Hom}_{A_s}(M_s,N_s) \cong \text{Hom}_A([M_s], [N_s]) \cong \text{Hom}_A(M, [N_s]).
\]
\end{Lem}

\begin{proof}
(1) We prove the first isomorphism.
For an $A_s$-homomorphism $h_s: M_s \to N_s$ and for any $i \in I$,  we assign
\[
[h_s]_i = h_s \otimes_{A_s} A_p: M_s \otimes_{A_s} A_p \to N_s \otimes_{A_s} A_p  
\]
for $p = i \cap s$.
Conversely, $h: [M_s] \to [N_s]$ gives $h_s: M_s \to N_s$ because $[M_s]_s = M_s$ and $[N_s]_s = N_s$.

(2) The second isomorphism is given by combining with $r^M_s: M \to [M_s]$.
Conversely, $h: M \to [N_s]$ gives $h_s: M_s \to [N_s]_s = N_s$.
\end{proof}

\begin{Cor}
(1) If $M_s$ is an injective $A_s$-module, then $[M_s]$ is an injective module in $\text{Qcoh}(A,\phi)$.

(2) Let $M \in \text{Qcoh}(A,\phi)$ and $N_s \in \text{Mod-}A_s$.
If $M$ is locally free, then 
\[
\text{Ext}_A^p(M, [N_s]) \cong 0
\]
for $p > 0$.
\end{Cor}

\begin{proof}
(1) If $N \to N'$ is an injective homomorphism in $\text{Qcoh}(A,\phi)$, then 
$N_s \to N'_s$ is an injective right $A_s$-homomorphism.
Since $M_s$ is an injective $A_s$-module, 
$\text{Hom}_{A_s}(N'_s, M_s) \to \text{Hom}_{A_s}(N_s, M_s)$ is surjective, hence
$\text{Hom}_A(N', [M_s]) \to \text{Hom}_A(N, [M_s])$ is also surjective.
Therefore $[M_s]$ is an injective module.

(2) We embed $N_s$ to an injective right $A_s$-module $N'_s$, and construct an 
exact sequence $0 \to N_s \to N'_s \to N''_s \to 0$.
Since $M_s$ is a free module of finite rank, we have an exact sequence
$0 \to \text{Hom}_{A_s}(M_s, N_s) \to \text{Hom}_{A_s}(M_s, N'_s) \to \text{Hom}_{A_s}(M_s, N''_s) \to 0$.
Then we have an exact sequence
\[
0 \to \text{Hom}_A(M, [N_s]) \to \text{Hom}_A(M, [N'_s]) \to \text{Hom}_A(M, [N''_s]) \to 0.
\]
Since $[N'_s]$ is injective, we have $\text{Ext}_A^p(M, [N'_s]) = 0$ for $p > 0$.
Therefore we obtain $\text{Ext}_A^1(M, [N_s]) = 0$.
Since $N_s$ was arbitrary, we have also $\text{Ext}_A^1(M, [N''_s]) = 0$.
It follows that $\text{Ext}_A^2(M, [N_s]) = 0$.
Repeating this argument, we obtain our assertion.
\end{proof}

Let $d = \# I$.
We assign numbers to elements in $I$ such that 
$I = \{i(1),\dots,i(d)\}$ in order to express the result below.
We note that these numbers are irrelevant to the order of $I$, i.e., 
$j_1 < j_2$ does not imply $i(j_1) < i(j_2)$.

\begin{Thm}
Let $M$ be a quasi-coherent module on an NC scheme $A = (A,\phi)$.
Then there is an exact sequence in $\text{Qcoh}(A, \phi)$:
\[
\begin{split}
&0 \to M \to \bigoplus_{1 \le j_0 \le d} [M_{i(j_0)}] \to \bigoplus_{1 \le j_0 < j_1 \le d} [M_{i(j_0) \cap i(j_1)}] \\
&\to \bigoplus_{1 \le j_0 < j_1 < j_2 \le d} [M_{i(j_0) \cap i(j_1) \cap i(j_2)}] \to \dots 
\end{split}
\]
where the arrows are given by 
\[
d_M^{\,p} = \bigoplus_{1 \le j_0 < \dots < j_p \le d} \,\, \sum_{k = 0}^p (-1)^k
r_{i(j_k)}^{[M_{\bigcap_{l \ne k} i(j_l)}]}.
\]
\end{Thm}

For example, $d_M^{\,0} = \bigoplus_{1 \le j_0 \le d} r_{i(j_0)}^M$, 
$d_M^{\,1} = \bigoplus_{1 \le j_0 < j_1 \le d} (r_{i(j_0)}^{[M_{i(j_1)}]} - r_{i(j_1)}^{[M_{i(j_0)}]})$, 
$d_M^{\,2} = \bigoplus_{1 \le j_0 < j_1 < j_2 \le d} (r_{i(j_0)}^{[M_{i(j_1) \cap i(j_2)}]} 
- r_{i(j_1)}^{[M_{i(j_0) \cap i(j_2)}]} + r_{i(j_2)}^{[M_{i(j_0) \cap i(j_1)}]})$, \dots.

\begin{proof}
The sequence is certainly well-defined and becomes a complex. 
We will check the exactness locally.
Thus we will prove that the following sequence is exact for each fixed $i(r) \in I$:
\[
\begin{split}
&0 \to M_{i(r)} \to \bigoplus_{1 \le j_0 \le d} M_{i(j_0) \cap i(r)} 
\to \bigoplus_{1 \le j_0 < j_1 \le d} M_{i(j_0) \cap i(j_1) \cap i(r)} \to \dots \\
&\to \bigoplus_{1 \le j_0 < \dots < j_p \le d} M_{i(j_0) \cap \dots \cap i(j_p) \cap i(r)} \to \dots 
\end{split}
\]
We can rewrite it to
\[
\begin{split}
&0 \to M_{i(r)} \to M_{i(r)} \oplus \bigoplus_{j_0 \ne r} M_{i(j_0) \cap i(r)} \\
&\to \bigoplus_{r < j_1} M_{i(r) \cap i(j_1)} \oplus \bigoplus_{j_0 < r} M_{i(j_0) \cap i(r)} 
\oplus \bigoplus_{j_0 < j_1,\, j_0 \ne r,\, j_1 \ne r} M_{i(j_0) \cap i(j_1) \cap i(r)} \\
&\to \bigoplus_{r < j_1 < j_2} M_{i(r) \cap i(j_1) \cap i(j_2)} \oplus 
\bigoplus_{j_0 < r < j_2} M_{i(j_0) \cap i(r) \cap i(j_2)} \oplus 
\bigoplus_{j_0 < j_1 < r} M_{i(j_0) \cap i(j_1) \cap i(r)} \\
&\oplus \bigoplus_{j_0 < j_1 < j_2, \,j_0 \ne r,\, j_1 \ne r,\, j_2 \ne r} M_{i(j_0) \cap i(j_1) \cap i(j_2) \cap i(r)} 
\to \dots \\
&\to \bigoplus_{q = 0}^p \bigoplus_{j_0 < j_1 < \dots < j_p,\, j_q = r} M_{i(j_0) \cap \dots \cap i(j_p)} 
\oplus \bigoplus_{j_0 < j_1 < \dots < j_p, \,j_q \ne r, \,\forall q} M_{i(j_0) \cap \dots \cap i(j_p) \cap i(r)} 
\to \dots 
\end{split}
\]
The initial two terms of $M_{i(r)}$ cancel when we take the cohomology.
The terms $\bigoplus_{j_0 \ne r} M_{i(j_0) \cap i(r)}$ and 
$\bigoplus_{r < j_1} M_{i(r) \cap i(j_1)} \oplus \bigoplus_{j_0 < r} M_{i(j_0) \cap i(r)}$ cancel again, and so on.
Hence the sequence is exact.
\end{proof}

\begin{Defn}
We define a \v Cech complex of a quasi-coherent module $M$ by 
\[
\check C^{\,p}(M) = \bigoplus_{j_0 < \dots < j_p} [M_{i(j_0) \cap \dots \cap i(j_p)}], \,\,\,0 \le p < d,
\]
\[
d_M^{\,p}(x) = \bigoplus_{j_0 < \dots < j_p} \sum_{k = 0}^p (-1)^k
r_{i(j_k)}^{[M_{\bigcap_{l \ne k} i(j_l)}]}, \,\,\, 0 < p < d.
\]
\end{Defn}

\begin{Cor}\label{Cech}
There is a spectral sequence
\[
E_1^{p,q} = \text{Ext}^q_A(M, \check C^p(N)) \Rightarrow \text{Ext}_A^{p+q}(M,N)
\]
for $M,N \in \text{Qcoh}(A,\phi)$.
In particular, if $M$ is locally free of finite rank, then
\[
H^p(\text{Hom}_A(M, \check C^{\bullet}(N)) \cong \text{Ext}_A^p(M,N).
\]
\end{Cor}

%%%%%%%%%%%%%%%%%%%%%%%%%%%%%%%%%%
%%%%%%%%%%%%%%%%%%%%%%%%%%%%%%%%%%
%%%%%%%%%%%%%%%%%%%%%%%%%%%%%%%%%%
\section{extension of locally free modules and tilting bundles}

Let $X^0$ be an algebraic variety with an affine open covering $U_i = \text{Spec}(A^0_i)$.
We assume that the index set $I = \{i\}$ has the minimum for any two elements, i.e., 
for $i,j \in I$, the intersection $U_i \cap U_j$ is equal to $U_k$ for some $k \in I$.
Let $\phi^0_{ij}$ be the gluing homomorphisms for the $A^0_i$.

First we consider the extension problem of locally free modules along NC deformations:
 
\begin{Thm}\label{extend F}
Let $X^0$ as above and let $F^0$ be a locally free sheaf of rank $r$ on $X^0$ with the gluing isomorphisms $\psi^0_{ij}$.
Let $0 \to J \to R' \to R \to 0$ be a small extension of Artin local rings, i.e., $\mathfrak m' J = 0$, 
let $A' = (A'_i, \phi'_{ij})$ be an NC deformation of $A^0 = (A^0_i, \phi_{ji}^0)$ 
over an Artin local ring $(R', \frak m')$, 
let $A = (A_i, \phi_{ij})$ be an NC deformation over $(R, \frak m)$ 
induced from $A'$, i.e., $A = R \otimes_{R'} A'$,
and let $F = (F_i, \psi_{ij})$ be a locally free right $A$-module of rank $r$ such that
$F \otimes_A A^0 \cong F^0$.
Then the following hold:

(1) There is an obstruction class $\xi \in J \otimes H^2(X^0, \mathcal End(F^0))$ such that $\xi = 0$ holds if and only if 
there exists a locally free right $A'$-module $F' = (F'_i, \psi'_{ij})$ such that 
$F' \otimes_{A'} A \cong F$. 

(2) In the case $\xi = 0$, the set of all such extensions $F'$ up to isomorphisms 
which induce the identity on $F$ is a torsor over $J \otimes H^1(X^0, \mathcal End(F^0))$.
\end{Thm}

\begin{proof}
(1) We express $F$ as $F_i = A_i^{\oplus r}$ with gluing isomorphisms 
$\psi_{ij}: F_j \otimes_{A_j} A_i \to F_i$ 
for $i < j$ which satisfy the cocycle condition.
Let $F'_i = (A'_i)^{\oplus r}$, 
and let $\psi'_{ij}: F'_j \otimes_{A'_j} A'_i \to F'_i$ for $i < j$ 
be any isomorphisms
which induce the $\psi_{ij}$ under the base change to $A$.
There exists an extension $F'$ if and only if there is a choice of the 
$\psi'_{ij}$ which satisfies the cocycle condition:
\[
\psi'_{ij}\psi'_{jk}(m) = \psi'_{ik}(m): F'_k \otimes_{A'_k} A'_i \to F'_i
\]
for $i < j < k$.
Therefore we define $\delta_{ijk} \in J \otimes \text{Hom}_{A^0_i}(F^0_k \otimes_{A^0_k} A^0_i, F^0_i)$ by
\[
\delta_{ijk}(m) = \psi'_{ij}\psi'_{jk}(m) - \psi'_{ik}(m) 
\]
for $i < j < k$.

$\{\delta_{ijk}\}$ is a closed $2$-cochain in the following sense:
\[
\begin{split}
&\psi^0_{ij} \delta_{jkl}(m) - \delta_{ikl}(m) + \delta_{ijl}(m) - \delta_{ijk}\psi^0_{kl}(m) \\
&= \psi'_{ij} (\psi'_{jk}\psi'_{kl}(m) - \psi'_{jl}(m)) - (\psi'_{ik}\psi'_{kl}(m) - \psi'_{il}(m)) \\
&+ (\psi'_{ij}\psi'_{jl}(m) - \psi'_{il}(m)) - (\psi'_{ij}\psi'_{jk}\psi'_{kl}(m) - \psi'_{ik}\psi'_{kl}(m)) = 0
\end{split}
\]
where we used $J (\psi'_{ij} - \psi^0_{ij}) = 0$ and the cocycle condition for the $\psi_{ij}$.
Then we can extend $\{\delta_{ijk}\}$ to a closed $2$-cochain for all $i,j,k$ by \cite{smooth} Lemma 2.7, 
and we define an obstruction class $\xi = [\delta_{ijk}] \in J \otimes H^2(X, \mathcal End(F^0))$.

We can change the gluing isomorphisms 
by a $1$-cochain $\epsilon_{ij} \in J \otimes \text{Hom}_{A^0_i}(F^0_j \otimes_{A^0_j} A^0_i, F^0_i)$:
\[
\psi'_{ij} \mapsto \psi'_{ij} + \epsilon_{ij}.
\]
Then the $2$-cocycle $\{\delta_{ijk}\}$ changes as follows:
\[
\begin{split}
&\delta_{ijk}(m) \mapsto (\psi'_{ij} + \epsilon_{ij}) (\psi'_{jk} + \epsilon_{jk})(m) - (\psi'_{ik} + \epsilon_{ik})(m) \\
&= \delta_{ijk}(m) + \epsilon_{ij} \psi^0_{jk} (m) - \epsilon_{ik}(m) + \psi^0_{ij}\epsilon_{jk}(m).
\end{split}
\]
That is, the $2$-cocycle $\{\delta_{ijk}\}$ changes by the coboundary of the $1$-cochain $\{\epsilon_{ij}\}$. 
Therefore the existence of the extension $F'$ is equivalent to the vanishing of the cohomology class $\xi$.

(2) We assume that $\xi = 0$ and there is one extension $F'$ of $F$.
Then we have $\delta_{ijk} = 0$ for the corresponding choice of the $\psi'_{ij}$.
When we change $\psi'_{ij}$ to $(\psi'_{ij})^1 = \psi'_{ij} + \epsilon_{ij}$, the $\delta_{ijk}$ stay zero if and only if
$\{\epsilon_{ij}\}$ is a closed $1$-cochain.
In the case where $\{\epsilon_{ij}\}$ is a closed $1$-cochain, let $F'_1$ be the corresponding extension. 
There is an isomorphism $F' \to F'_1$ inducing the identity on $F$, if and only if 
there are isomorphisms $F'_i \to (F'_1)_i$ for all $i$ given by $1 + \gamma_i$ for
$\gamma_i \in J \otimes \text{Hom}_{A^0_i}(F^0_i, F^0_i)$ such that
$(1 + \gamma_i) \psi'_{ij} =  (\psi'_{ij} + \epsilon_{ij}) (1 + \gamma_j)$, i.e.,
\[
\epsilon_{ij} = \gamma_i \psi^0_{ij} - \psi^0_{ij}\gamma_j. 
\]
Thus the extension corresponding to a $1$-cocycle $\{\epsilon_{ij}\}$ are isomorphic if and only if 
$\{\epsilon_{ij}\}$ is a coboundary.
Therefore the isomorphism classes of the extensions are classified by the first cohomology classes.
\end{proof}

Next we consider the extensions of pretilting bundles and their endomorphism algebras along NC deformations:

\begin{Thm}
Let $F^0$ be a locally free sheaf of rank $r$ on $X^0$, and 
let $A = (A_i, \phi_{ji})$ be an NC deformation of $A^0 = (A_i, \phi_{ji}^0)$ 
over an Artin local ring $(R, \mathfrak m)$.
Assume that $F^0$ is {\em pretilting}, i.e., $\text{Ext}_{X^0}^p(F^0,F^0) = 0$ for $p > 0$.
Then there is a locally free right $A$-module $F$ which satisfies the following conditions:

(1) $R/\frak m \otimes_R F \cong F^0$, and this property determines $F$ uniquely 
up to isomorphisms inducing the identity on $F^0$.

(2) $\text{End}_A(F)$ is a flat $R$-module and 
$R/\frak m \otimes_R \text{End}_A(F) \cong \text{End}_{X^0}(F^0)$.

(3) $\text{Ext}_A^p(F,F) = 0$ for $p > 0$.
\end{Thm}

\begin{proof}
(1) This is a consequence of Theorem \ref{extend F}.

We will prove (2) and (3) together.
By Corollary \ref{Cech}, we have 
\[
\text{Ext}_A^p(F,F) \cong H^p(\text{Hom}_A(F, \check C^{\bullet}(F))) 
\cong H^p(K^{\bullet})
\]
where $K^{\bullet}$ is a complex of $R$-modules whose $p$-th term is
\[
K^p = \bigoplus_{j_0 < \dots < j_p} \text{End}_{A_{i(j_0) \cap \dots \cap i(j_p)}} (F_{i(j_0) \cap \dots \cap i(j_p)})
\]
because $\text{Hom}_A(F, [F_i]) \cong \text{Hom}_{A_i}(F_i,F_i)$.
If the rank of $F_i$ is $r$, then it is an $R$-algebra of matrices $M(r, A_i)$, the algebra of $r \times r$ matrices
with entries in $A_i$.
Thus we have
\[
R/\frak m \otimes_R K^p \cong K^p_0 :=  \bigoplus_{j_0 < \dots < j_p} \text{End}_{A^0_{i(j_0) \cap \dots \cap i(j_p)}}
(F^0_{i(j_0) \cap \dots \cap i(j_p)}).
\]
We know that $H^p(K^{\bullet}_0) \cong H^p(X, \mathcal End_{X^0}(F^0)) = 0$ for $p > 0$ by the assumption.

We have a spectral sequence:
\[
E_2^{p,q} = \text{Tor}_{-p}^R(R/\frak m, H^q(K^{\bullet})) \Rightarrow H^{p+q}(K_0^{\bullet}).
\]
We will prove that $H^q(K^{\bullet}) = 0$ for $q > 0$ by the descending induction on $q$, noting that  
$K^{\bullet}$ is a bounded complex, because $I$ is a finite set, and $E_2^{p,q} = 0$ for $p > 0$.

Assuming that we have already $H^q(K^{\bullet}) \cong 0$ if $q > q_1$ for an integer $q_1 > 0$, 
we will prove that 
$H^{q_1}(K^{\bullet}) \cong 0$.
Indeed we have $E_2^{p,q} = 0$ for $q > q_1$, hence $E_2^{0,q_1} = 0$ because $H^{q_1}(K_0^{\bullet}) = 0$.
Thus we have $R/\frak m \otimes_R H^{q_1}(K^{\bullet}) = 0$.
It follows that $H^{q_1}(K^{\bullet}) \cong 0$ by Nalayama's Lemma.
Thus we have (3).
Then we have $E_2^{p,0} = 0$ for $p < 0$ and $E_2^{0,0} \cong \text{End}_{A^0}(F^0)$, hence (2).
\end{proof}

We also prove that the generation property extends under NC deformations:

\begin{Lem}
Let $A = (A_i, \phi_{ij})$ be an NC deformation of an algebraic variety $X^0$ over an Artin local ring $(R, \frak m)$
as before, 
let $F$ be a locally free coherent module on $A$, and let $M$ be a bounded above complex 
of coherent modules on $A$ such that $M \not\cong 0$ in $D^-\text{coh}(A)$.
Assume that, for any $N \in D^-\text{coh}(A^0)$, $N \not\cong 0$ implies that 
$R\text{Hom}_{A^0}(R/\frak m \otimes_R F, N) \not\cong 0$.
Then $R\text{Hom}_A(F, M) \not\cong 0$.
\end{Lem}

\begin{proof}
Since the \v Cech resolution have fixed finite length, we may replace $M$ by its \v Cech resolution, so that the terms $M^p$ of $M$ are 
finite direct sums of coherent modules of the forms $[L_i]$.
By taking $A_i$-free resolutions, we may assume that the $L_i$ are free modules of finite ranks, because the 
operation $[ \,\,\,]$ is exact.
It follows in particular that the $M^p$ are free $R$-modules.

Since $M \not\cong 0$, there exists an integer $p_0$ such that $H^{p_0}(M) \not\cong 0$.
We take the largest such $p_0$, noting that $M$ is bounded above.
Let $N = R/\frak m \otimes_R^L M \in D^-\text{coh}(A^0)$.
Since $M^p$ are $R$-free, $N$ consists of terms $N^p = R/\frak m \otimes_R M^p$, and
we have $H^{p_0}(N) \cong R/\frak m \otimes_R H^{p_0}(M) \not\cong 0$.
Thus $N \not\cong 0$.

By the assumption, we have $R\text{Hom}_{A^0}(F^0, N) \not\cong 0$ for $F^0 = R/\frak m \otimes_R F$.
This means that $H^{p_1}(\text{Hom}_{A^0}(F^0, N^{\bullet})) \ne 0$ for some $p_1$.
We take the largest such $p_1$.
We have $\text{Hom}_{A^0}(F^0, N^p) \cong R/\frak m \otimes_R \text{Hom}_A(F, M^p)$
because $\text{Hom}_A(F, [L_i]) \cong A_i^{\oplus rs}$ as $R$-modules for $r = \text{rank}(F)$ and $s = \text{rank}(L_i)$. 
Therefore $H^{p_1}(\text{Hom}_A(F, M^{\bullet})) \ne 0$, hence
$R\text{Hom}_A(F, M) \not\cong 0$.
\end{proof}

%%%%%%%%%%%%%%%%%%%%%%%%%%%%%%%%%%
%%%%%%%%%%%%%%%%%%%%%%%%%%%%%%%%%%
%%%%%%%%%%%%%%%%%%%%%%%%%%%%%%%%%%
\section{derived McKay correspondence}

First we recall the definition of the tilting bundles and the Bondal-Rickard equivalence (\cite{TU}).
Let $f^0: X^0 \to \text{Spec}(B^0)$ be a proper morphism from an algebraic variety to an affine variety over $k$.
A locally free coherent sheaf $F^0$ on $X^0$ is said to be a {\em tilting bundle} if the
following conditions are satisfied:

(a) (pretilting) $\text{Ext}^p_{X^0}(F^0, F^0) = 0$ for $p > 0$.

(b) (generator) For any $x \in D^-\text{coh}(X^0)$, if $x \not\cong 0$, then $R\text{Hom}_{X^0}(F^0, x) \not\cong 0$.

Let $E^0 = \text{Hom}_{X^0}(F^0, F^0)$ be the endomorphism algebra.
Then there is a $B^0$-linear equivalence of triangulated categories:
\[
\Phi: D^b\text{coh}(X^0) \to D^b\text{mod-}E^0
\]
given by $\Phi(x) = R\text{Hom}_{X^0}(F^0, x)$ with its quasi-inverse given by $\Psi(y) = y \otimes^{L}_{E^0} F^0$.

We will prove that the Bondal-Rickard equivalence extends under NC deformations:

\begin{Thm}
Let $f^0: X^0 \to B^0$ and $F^0$ be as above.
We cover $X^0$ by a finite number of affine open subsets.
Let $(R, \frak m)$ be an Artin local algebra with $R/\frak m = k$, and let $A = (A_i,\phi_{ij})$ be an NC deformation of $X^0$
over $R$.

(1) Let $F$ be a locally free coherent module on $A$ which is an extension of $F^0$.
Then $F$ is a tilting bundle in the following sense:

\hskip 1pc (a) $R^p\text{Hom}_A(F, F) = 0$ for $p > 0$.

\hskip 1pc (b) For $x \in D^-\text{coh}(A)$, if $x \not\cong 0$, then $R\text{Hom}_A(F, x) \not\cong 0$.

(2) Let $E = \text{Hom}_A(F, F)$ be the endomorphism algebra.
Then there is a $k$-linear equivalence of triangulated categories:
\[
\Phi: D^b\text{coh}(A) \to D^b\text{mod-}E
\]
given by $\Phi(x) = R\text{Hom}_A(F, x)$ with its quasi-inverse given by $\Psi(y) = y \otimes^{L}_E F$.
\end{Thm}

\begin{proof}
(1) is already proved in the previous section.

(2) is proved by the same argument as in \cite{TU}.
By the lemma below, $E$ is noetherian.
It follows that any finitely generated right $E$-module has a resolution by finitely generated free modules.
Hence a functor $\Psi: D^-\text{mod-}E \to D^-\text{coh}(A)$ is defined by $\Psi(y) = y \otimes^{L}_E F$.
On the other hand, $\Phi(x)$ is bounded if $x$ is bounded, 
because the \v Cech complex is bounded thanks to the finiteness of the affine covering.
Therefore the same argument as in \cite{TU} works.
\end{proof}

\begin{Lem}
$E$ is a noetherian $k$-algebra.
\end{Lem}

\begin{proof}
We will prove that $E$ is (right) noetherian by induction on the length of $R$.
If $R = k$, then $E^0$ is finitely generated as a $B^0$-module, hence noetherian.
Let $0 \to J \to R' \to R \to 0$ be a small extension, and assume that $E = R \otimes_{R'} E'$ is noetherian.
Let $L' \subset E'$ be a right ideal, and let $L = L'E \subset E$.
Then we have an exact sequence
$0 \to L' \cap JE' \to L' \to L \to 0$.
By the assumption, $L$ is a finitely generated right ideal of $E$.
We lift the generators to elements of $L'$.
On the other hand, $L' \cap JE' \subset J \otimes_k E^0$ is also finitely generated.
By combining the generators, we conclude that $L'$ is finitely generated. 
\end{proof}

%%%%%%%%%%%%%%%%%%%%%%%%%%%%%%%%%%%%%%%%%%
%%%%%%%%%%%%%%%%%%%%%%%%%%%%%%%%%%%%%%%%%%
%%%%%%%%%%%%%%%%%%%%%%%%%%%%%%%%%%%%%%%%%%

Graduate School of Mathematical Sciences, University of Tokyo,
Komaba, Meguro, Tokyo, 153-8914, Japan. 

kawamata@ms.u-tokyo.ac.jp

\end{document}